\documentclass[12pt,oneside]{amsart}

\usepackage[english]{babel}
\usepackage{cite}
\usepackage{latexsym}
\usepackage{amssymb,amsthm,amsmath}
\usepackage{hyperref}
\usepackage{graphicx}
\usepackage{longtable}
\usepackage{xcolor}
\usepackage{enumerate}
\usepackage{mathdots}

\newtheorem{theorem}{Theorem}[section]
\newtheorem{proposition}[theorem]{Proposition}

\newtheorem{corollary}[theorem]{Corollary}
\newtheorem{remark}[theorem]{Remark}

{\theoremstyle{definition}
\newtheorem*{definition*}{Definition}
\newtheorem{definition}[theorem]{Definition}}
\newtheorem*{proposition*}{Proposition}
\newtheorem*{corollary*}{Corollary}
\newtheorem*{lemma*}{Lemma}
\newtheorem*{remark*}{Remark}

\definecolor{dgreen}{rgb}{0.13,0.7,.63}

\def\daniele#1 {\fbox {\footnote {\ }}\ \footnotetext { From Daniele: {\color{blue}#1}}}

\def\matteo#1 {\fbox {\footnote {\ }}\ \footnotetext { From Matteo: {\color{teal}#1}}}
\title{Minimal linear codes in odd characteristic}

\author{Daniele Bartoli}
\address{Department of Mathematics and Informatics, University of Perugia, Perugia,  Italy}
\email {daniele.bartoli@unipg.it}

\author{Matteo Bonini}
\address{Department of Mathematics, University of Trento, Trento, Italy}
\email {matteo.bonini@unitn.it}

\date{}

\begin{document}

\begin{abstract}
In this paper we generalize  constructions in two recent works of Ding, Heng, Zhou to any field $\mathbb{F}_q$, $q$ odd, providing infinite families of minimal codes for which the Ashikhmin-Barg bound  does not hold.
\end{abstract}

\maketitle


\section{Introduction}
Let $\mathcal{C}$ be a linear code. A codeword $c\in \mathcal{C}$ is said a \emph{minimal codeword} if its support (i.e. the set of non-zero coordinates) determines $c$ up to a scalar factor. Equivalently, the support of $c$ does not contain the support of any other independent codeword.

Minimal codewords can be used \cite{Massey1993,Massey1995} in linear codes-based access structures in secret sharing schemes (SSS), that is  protocols which include a distribution algorithm and a reconstruction algorithm, implemented by a dealer and some participants; see \cite{Shamir1979,Blakley1979}. The  dealer  splits a secret $s$ into different pieces (shares) and distributes them to participants $\mathcal{P}$. Only  authorized subsets of $\mathcal{P}$ (access structure $\Gamma$) can be able to reconstruct the secret by using their respective shares.  A set of participants $A$ is called a \emph{minimal authorized subsets} if $A\in \Gamma$ and no proper subset of $A$ belongs to $\Gamma$. An SSS is called \emph{perfect} if only authorized sets of participants can recover the secret and \emph{ideal} if the shares are of the same size as that of the secret.

In his works Massey \cite{Massey1993,Massey1995} used linear codes for a perfect and ideal SSS. Also, he pointed out the relationship between the access structure and the set of minimal codewords of the dual code of the underlying code. In particular, the access structure of the secret-sharing scheme corresponding to an $[n,k]_q$-code $\mathcal{C}$ is specified by the support of minimal codewords in $\mathcal{C}^{\bot}$ having $1$ as first component; see  \cite{Massey1993,Massey1995}.

Given an arbitrary linear code $\mathcal{C}$, it is a hard task to determine the set of its minimal codewords even	in the binary case.  In fact, the knowledge of the minimal codewords is related with the complete decoding problem, which is a NP-problem  even if preprocessing is allowed \cite{BMeT1978,BN1990}; this means that to obtain the access structures of the SSS based on general linear codes is also hard. In  general this has been done only for specific classes of linear codes and this led to the study of linear codes for which every codeword is minimal; see for instance \cite{CCP2014,SL2012}.

Ashikhmin and Barg \cite{AB1998} gave a useful criterion for a linear code to be minimal.
\begin{theorem}\label{Th:Ashikhmin-Barg}
A linear code $\mathcal{C}$ over $\mathbb{F}_q$ is minimal if 
\begin{equation}\label{Eq:Ashikhmin-Barg}
\frac{w_{min}}{w_{max}} > \frac{q-1}{q},
\end{equation}
where $w_{min}$ and $w_{max}$ denote the minimum and maximum nonzero Hamming weights in $\mathcal{C}$.
\end{theorem}

On the one hand,  families of minimal linear codes satisfying Condition \eqref{Eq:Ashikhmin-Barg} have been considered in for instance \cite{CDY2005,Ding2015,DLLZ2016,YD2006}. On the other hand, Condition \eqref{Eq:Ashikhmin-Barg} is not necessary for linear codes to be
minimal. In this direction,  sporadic examples of minimal codes have been presented in \cite{CMP2013}, whereas in \cite{CH2017} the first infinite family of minimal binary codes has been constructed by means of Boolean functions arising from  simplicial complexes.  More recently, families of minimal binary and ternary codes have been investigated in \cite{HDZ2018,DHZ2018}.

In this paper we generalize the constructions in \cite{HDZ2018,DHZ2018} to any field $\mathbb{F}_q$, $q$ odd, providing infinite families of minimal linear codes for which Condition \eqref{Eq:Ashikhmin-Barg} does not hold. 

\section{Minimal codes and Secret Sharing Schemes}
Let $\mathcal{C}$ be an $[n,k]_q$-code, that is a $k$-dimensional linear subspace of $\mathbb{F}_q^n$. The support $Supp(c)$ of a codeword $c=(c_1,\ldots,c_n)\in \mathcal{C}$ is the set $\{i \in [1,\ldots,n] \ : \ c_i \neq 0\}$. Clearly, the Hamming weight $w(c)$ equals $|Supp(c)|$ for any codeword $c\in \mathcal{C}$. 
\begin{definition}\cite{Massey1993} A codeword $c\in \mathcal{C}$ is \emph{minimal} if it only covers the codewords $\lambda c$, with $\lambda \in \mathbb{F}_q^*$, that is 
$$\forall \ c^{\prime}\in \mathcal{C} \Longrightarrow  \left(Supp(c) \subset Supp(c^{\prime}) \Longrightarrow \exists \lambda \in \mathbb{F}_q \ : \ c^{\prime}=\lambda c\right).$$
\end{definition}

\begin{definition}\cite{DY2003}
The code $\mathcal{C}$ is \emph{minimal} if every non-zero codeword $c \in \mathcal{C}$ is minimal.
\end{definition}

Let $G \in \mathbb{F}_q^{k\times n}$ be the generator matrix of $\mathcal{C}$ with columns $G_1,\ldots,G_n$ and suppose that no $G_i$ is the $0$-vector. The code $\mathcal{C}$ can be used to construct secret sharing schemes in the following way. The secret is an element of $\mathbb{F}_q$ and the set of participants  $\mathcal{P}=\{P_2,\ldots,P_{n}\}$. The dealer chooses randomly $u = (u_1, \ldots, u_{k}) \in \mathbb{F}_q^k$ such that  $s = u\cdot G_1$ and  computes the corresponding codeword $v=(v_1,\ldots,v_n)= u G$. Each participant $P_i$, $i\geq 2$,  receives the share $v_i$. A set of participants $\{P_{i_1},\ldots,P_{i_\ell}\}$ determines the secret if and only if $G_1$ is a linear combination of $G_{i_1},\ldots,G_{i_{\ell}}$; see \cite{Massey1993}. There is a one-to-one correspondence between minimal authorized subsets and the set of minimal codewords of the dual code $\mathcal{C}^{\bot}$.

\section{A family of minimal codes violating the Ashikhmin-Barg bound}
\subsection{Notations and definition of the code \texorpdfstring{$\mathcal{C}_f$}{Lg}}
Let $q=p^h$, $p$ odd prime, $h\geq 1$, and consider the Galois field $\mathbb{F}_q$. Fix an integer $m>3$ and consider $k\in [2,\ldots,m-2]$. Choose $\alpha_i$, $i=1,\ldots,k$, to be (not necessarily distinct) elements of $\mathbb{F}_q^*$. Let us denote $(0,\ldots,0)\in \mathbb{F}_q^m$ by $\overline{0}$.

The weight $w(x)$ of a vector $x=(x_1,\ldots,x_m)\in \mathbb{F}_q^m$ is defined as $| \{ i \in [1,\ldots,m] \ : \ x_i\neq 0\}|$. 

Consider the function $f\ : \ \mathbb{F}_q^m \setminus\{\overline{0}\} \to \mathbb{F}_q$ defined by 
\begin{equation}\label{Def:f4}
f(x) = \left\{
\begin{array}{ll}
\alpha_i, & w(x)=i,\\
0,& w(x)> k,\\
\end{array}
\right.
\end{equation}
for any $x\in \mathbb{F}_{q^m}\simeq \mathbb{F}_q^m$, $x\neq 0$. 

We define the code $\mathcal{C}_f$ as
\begin{equation}\label{Def:Code4}
\mathcal{C}_f =\{(uf(x)+v\cdot x)_{x \in \mathbb{F}_q^m \setminus \{\overline{0}\}} \ : \ u \in \mathbb{F}_q, v \in \mathbb{F}_q^m\},
\end{equation}
where $v\cdot x$ denotes the usual inner product in $\mathbb{F}_q^m$ between $v=(v_1,\ldots,v_m)$ and $x =(x_1,\ldots,x_m)$. 

As a notation, for any pair $(u,v) \in \mathbb{F}_q\times \mathbb{F}_{q}^m$ let $c(u,v)=(uf(x)+v\cdot x)_{x \in \mathbb{F}_q^m \setminus \{\overline{0}\}}$ denote the corresponding codeword of $\mathcal{C}_f$. Choose any ordering in $\mathbb{F}_q^m \setminus \{\overline{0}\}$. For an $x\in\mathbb{F}_q^m \setminus \{\overline{0}\}$, we denote by $c(u,v)_x$  the entry in $c(u,v)$ corresponding to $x$. The support $Supp(c(u,v))$ of a codeword $c(u,v)$ is defined as the set of $\{x \in \mathbb{F}_q^m \setminus \{\overline{0}\} \ : \ c(u,v)_x\neq 0\}$. 

Finally, let $\mathrm{AG}(m,q)$ be the affine space of dimension $m$ over the field $\mathbb{F}_q$. A hyperplane in $\mathrm{AG}(m,q)$ is an affine subspace of dimension $m-1$. For a more detailed introduction on affine spaces over finite fields we refer the reader to \cite{HirschBook}.

\subsection{The minimality of the code  \texorpdfstring{$\mathcal{C}_f$}{Lg}}
Observe that, for any fixed pair $(u,v) \in \mathbb{F}_q \times \mathbb{F}_q^m$, the elements $x \in \mathbb{F}_q^m \setminus \{\overline{0}\}$ for which the codeword $c(u,v)_x=0$ are contained in the union of $k+1$ hyperplanes $H(v)$ and $L_i(u,v)$, $i=1,\ldots,k$, defined by 
\begin{equation}\label{Def:Hyperplanes4}
H(v)=\Big\{ (y_1,\ldots,y_m ) \in \mathbb{F}_q^m \ : \ \sum _{j=1}^m v_jy_j =0\Big\}, \quad L_i(u,v)=\Big\{ (y_1,\ldots,y_m ) \in \mathbb{F}_q^m \ : \ \sum _{j=1}^m v_jy_j =-\alpha_iu\Big\}.
\end{equation}
More precisely, 
$$\overline{Supp}(c(u,v))=\mathbb{F}_q^m\setminus \left(Supp(c(u,v))\cup \{ \overline{0}\}\right)= \Big\{x \in \mathbb{F}_q^m \setminus \{\overline{0}\} \ : \ c(u,v)_x=0 \Big\}$$ equals $\overline{H}(v)\cup \bigcup_{i=1}^{k} \overline{L}_i(u,v)$, where 
\begin{eqnarray*}
\overline{H}(v)&=& \left\{ (y_1,\ldots,y_m ) \in H(v) \ : \ w(y_1,\ldots,y_m )>k\right\},\\
\overline{L}_i(u,v)&=&\left\{ (y_1,\ldots,y_m ) \in L_i(u,v) \ : \ w(y_1,\ldots,y_m )=i\right\}.\\
\end{eqnarray*}

\begin{proposition}\label{Prop:Hyperplanes4}
Let $H(v)$ and $H(v^{\prime})$, $v,v^{\prime}\neq \overline{0}$, be two distinct hyperplanes defined as in \eqref{Def:Hyperplanes4}. Then there exist $A,B\in \mathbb{F}_q^m$ with $w(A),w(B) >k$ such that $A\in H(v) \setminus H(v^{\prime})$ and $B \in H(v^{\prime}) \setminus H(v)$. 
\end{proposition}
\proof
It is enough to prove that, for any two distinct hyperplanes of type $H(z)$ and $H(z^{\prime})$, 
$$|\left\{ (y_1,\ldots,y_m ) \in H(z) \ : \ w(y_1,\ldots,y_m )>k\right\} | >q^{m-2}=|H(z) \cap H(z^{\prime})|.$$

In fact, for a given $v=(v_1,\ldots, v_m)$, we can suppose that $v_m=1$ and therefore $H(v)=\{(y_1,\ldots,y_m ) \in \mathbb{F}_q^m \ : \ y_m = -\sum _{j=1}^{m-1} v_j y_j\}$. So, 
$$\Big|\Big\{ (y_1,\ldots,y_m ) \in H(v) \ : \ w(y_1,\ldots,y_m )>k\Big\}\Big| $$
$$ \geq\Big|\Big\{ (y_1,\ldots,y_m ) \in H(v) \ : \ w(y_1,\ldots,y_{m-1} )>k, \ y_m=-\sum _{j=1}^{m-1} v_j y_j\Big\}\Big|$$
$$\geq \sum_{j=k+1}^{m-1}\binom{m-1}{j} (q-1)^j\geq (q-1)^{m-1}>q^{m-2}.$$
\endproof

\begin{theorem}\label{Th:Main4}
The code $\mathcal{C}_f$ is minimal.
\end{theorem}
\proof
Let $c(u,v)$ and $c(u^{\prime},v^{\prime})$ be two codewords, with $c(u,v)\neq \lambda c(u^{\prime},v^{\prime})$ for any $\lambda \in \mathbb{F}_q^*$, and both $c(u,v), c(u^{\prime},v^{\prime})$ different from the $0$-codeword. 

Suppose that ${Supp}(c(u^{\prime},v^{\prime}))\subset {Supp}(c(u,v))$, that is  $\overline{Supp}(c(u,v))\subset \overline{Supp}(c(u^{\prime},v^{\prime}))$. 
\begin{itemize}
\item Suppose $v=\overline{0}$. Then $u\neq 0$ and  $\overline{Supp}(c(u,v))$ consists of all $x\in \mathbb{F}_q^m$ with $w(x)>k$. Since $\overline{Supp}(c(u,v))\subset \overline{Supp}(c(u^{\prime},v^{\prime}))$, $v^{\prime}=\overline{0}$. It is easily seen that $c(u,\overline{0})=\lambda c(u^{\prime},\overline{0})$ for some $\lambda\in \mathbb{F}_q$, a contradiction.
\item Suppose $v^{\prime}=\overline{0}$. Then $u^{\prime}\neq 0$ and  $\overline{Supp}(c(u^{\prime},v^{\prime}))$ consists of all $x\in \mathbb{F}_q^m$ with $w(x)>k$. If $v\neq \overline{0}$ then $\overline{Supp}(c(u,v))$ would also contain some $x\in \mathbb{F}_q^m$ with $0<w(x)\leq k$, a contradiction to $\overline{Supp}(c(u,v))\subset \overline{Supp}(c(u^{\prime},v^{\prime}))$. So $v= \overline{0}$ and therefore  $c(u,\overline{0})=\lambda c(u^{\prime},\overline{0})$ for some $\lambda\in \mathbb{F}_q$, a contradiction.

\item Suppose $v,v^{\prime}\neq \overline{0}$. By Proposition \ref{Prop:Hyperplanes4}, $H(v)=H(v^{\prime})$, that is $v=\lambda v^{\prime}$ for some $\lambda \in \mathbb{F}_q^*$. Also, $L_i(u,v)\subset L_i(u^{\prime},v^{\prime})=L_i(u^{\prime},\lambda v)$, for any $i=1,\ldots,k$. Since $L_i(u,v)$ and $L_i(u^{\prime},\lambda v)$ can be either disjoint or coincident, $u^{\prime}=\lambda u$ and therefore $ c(u^{\prime},v^{\prime})=\lambda c(u,v)$, a contradiction.
\end{itemize}
Then ${Supp}(c(u^{\prime},v^{\prime}))\not \subset {Supp}(c(u,v))$ and $\mathcal{C}_f$ is minimal.
\endproof

\subsection{The parameters of  \texorpdfstring{$\mathcal{C}_f$}{Lg}}
\begin{proposition}\label{Prop:parameters}
The code $\mathcal{C}_f$ has length $q^m-1$ and dimension $m+1$ over $\mathbb{F}_{q}$. If 
\begin{equation}\label{Constraint1}
q^m-1-\sum_{i=1}^{m-1}\binom{m-1}{i}(q-1)^i -\sum_{i=1}^{k}\binom{m-1}{i}(q-1)^i \geq \sum_{i=1}^{k}\binom{m}{i}(q-1)^i
\end{equation}
then minimum and maximum weights in $\mathcal{C}_f$ satisfy
\[
w_{min}=\sum_{i=1}^{k}\binom{m}{i}(q-1)^i, \qquad w_{max}\geq q^{m}-q^{m-1}.
\]

Also, if 
\begin{equation}\label{Constraint2}
\sum_{i=1}^{k}\binom{m}{i}(q-1)^{i-1}\leq q^{m-1}-q^{m-2}
\end{equation}
then $w_{min}/w_{max}\leq (q-1)/q$.
\end{proposition}
\begin{proof}
Clearly, the length of $\mathcal{C}_f$ is $|\mathbb{F}_q^m\setminus \{\overline{0}\}|=q^m-1$.

Each codeword in $\mathcal{C}_f$ can be written as linear combination of $c(1,\overline{0})$, $c(0,e_1)$, \ldots, $c(0,e_m)$, where $e_1,\ldots,e_m$ is the standard basis of $\mathbb{F}_q^m$ over $\mathbb{F}_q$.  

On the other hand, suppose that $c(u,v)$ is the zero codeword. 
\begin{itemize}
\item If $u=0$, then for elements $y_i=e_i\in \mathbb{F}_q^m$, $i=1,\ldots,m$, we have $c(u,v)_{y_i}=v_i=0$, and then $v=\overline{0}$. 
\item If $u\neq 0$, then we can consider  $y_i=e_i$, $i=1,\ldots,m$ and $y=2e_i$ and then $c(u,v)_{y_i}=u\alpha_1+v_i=0$, $c(u,v)_{2y_i}=u\alpha_1+2v_i=0$. Since $\alpha_1\neq0$ (see \eqref{Def:f4}), the above conditions yield $v_1=\cdots=v_m=u=0$.
\end{itemize}
This proves that $c(1,\overline{0})$, $c(0,e_1)$, \ldots, $c(0,e_m)$ is a basis of $\mathcal{C}$ of size $m+1$. 

We now determine the minimum weight of the code. Recall that for a codeword $c(u,v)$ its weight is 

$$w(c(u,v))=\big| {Supp}((c(u,v))\big|=q^m-1-\big| \overline{Supp}((c(u,v))\big|.$$

\begin{itemize}
\item The codeword $c(0,\overline{0})$ is the 0-codeword.
\item The $q-1$ codewords $c(u,\overline{0})$, $u\neq 0$, have weight exactly $\sum_{i=1}^{k}\binom{m}{i}(q-1)^i$. In fact,  $c(u,\overline{0})_x=\alpha_{w(x)}u$ is non-zero if and only if $w(x)\in [1,\ldots,k]$. 
\item The $q^m-1$ codewords $c(0,v)$, $v\neq \overline{0}$, have weight exactly $q^m-q^{m-1}$, since each $x\in \mathbb{F}_q^m\setminus\{\overline{0}\}$ satisfying $v\cdot x=0$ belongs to $\overline{Supp}(c(0,v))$. 
\item For a codeword $c(u,v)$, with $u\neq0$ and $v\neq \overline{0}$, 
$$\overline{Supp}(c(u,v)) =\overline{H}(v)\cup \bigcup_{i=1}^{k} \overline{L}_i(u,v);$$
see Proposition \ref{Def:Hyperplanes4}. Without loss of generality we can suppose that $v_m=1$. We have that 

\begin{eqnarray*}
\sum_{i=k+1}^{m-1}\binom{m-1}{i}(q-1)^i= \left| \left\{ (x_1,\ldots,x_m) \ : \ x_m=-\sum_{j=1}^{m-1}v_jx_j, \ w(x_1,\ldots,x_{m-1})\geq k+1\right\}\right|&\leq\\
\big|\overline{H}(v)\big| \leq \left| \left\{ (x_1,\ldots,x_m) \ : \ x_m=-\sum_{j=1}^{m-1}v_jx_j, \ w(x_1,\ldots,x_{m-1})\geq k\right\}\right|&=\\
\sum_{i=k}^{m-1}\binom{m-1}{i}(q-1)^i.\\
\end{eqnarray*}
Analogously, 
\begin{eqnarray*}
0\leq \big|\overline{L}_i(u,v)\big| \leq \left| \left\{ (x_1,\ldots,x_m) \ : \ x_m=-f(x)u-\sum_{j=1}^{m-1}v_jx_j, \ w(x_1,\ldots,x_{m-1})\in [i-1,i]\right\}\right|&=\\
 \leq \binom{m-1}{i-1}(q-1)^{i-1}+\binom{m-1}{i}(q-1)^i.\\
\end{eqnarray*}
Thus, 
$$\sum_{i=k+1}^{m-1}\binom{m-1}{i}(q-1)^i \leq \big| \overline{Supp}(c(u,v)) \big|\leq \sum_{i=1}^{m-1}\binom{m-1}{i}(q-1)^i +\sum_{i=1}^{k}\binom{m-1}{i}(q-1)^i$$
and 
\begin{eqnarray*}
q^m-1-\sum_{i=1}^{m-1}\binom{m-1}{i}(q-1)^i -\sum_{i=1}^{k}\binom{m-1}{i}(q-1)^i   &\leq\\
 w(c(u,v)) \leq q^m-1-\sum_{i=k+1}^{m-1}\binom{m-1}{i}(q-1)^i.
\end{eqnarray*}

By \eqref{Constraint1}, the minimum weight is 
$$w_{min}=\sum_{i=1}^{k}\binom{m}{i}(q-1)^i,$$
whereas
$$w_{max}\geq q^{m}-q^{m-1}.$$
Finally, if \eqref{Constraint2} holds, 
$$\frac{w_{min}}{w_{max}}=\frac{\sum_{i=1}^{k}\binom{m}{i}(q-1)^i}{q^{m}-q^{m-1}}\leq\frac{q-1}{q}.$$

\end{itemize}

\end{proof}

\begin{remark}
Note that if \eqref{Constraint1} does not hold, then $w_{min}\leq\sum_{i=1}^{k}\binom{m}{i}(q-1)^i$. Arguing as in Proposition \ref{Prop:parameters}, Condition \eqref{Constraint2} yields $w_{min}/w_{max}\leq (q-1)/q$.

\end{remark}

\begin{corollary}
	If $q\geq5$, $2<m\le q-1$, and $k\le (m-1)/2$ then Conditions \eqref{Constraint1} and \eqref{Constraint2} hold.
\end{corollary}
\begin{proof}
First of all observe that 

\begin{align*}
\sum_{i=1}^{m-1} \binom{m-1}{i}(q-1)^i&=q^{m-1}-1,\\
\sum_{i=0}^{\alpha} \binom{m}{i}(q-1)^i\leq \sum_{i=0}^{\alpha} m^iq^i&\leq  2q^{2\alpha},  \\
\sum_{i=0}^{\alpha} \binom{m}{i}(q-1)^{i-1}\leq \sum_{i=0}^{\alpha} m^iq^{i-1}&\leq  2q^{2\alpha-1}.  \\
\end{align*}
Therefore we have that 
\begin{align*}
q^m-1-\sum_{i=1}^{m-1}\binom{m-1}{i}(q-1)^i -\sum_{i=1}^{k}\binom{m-1}{i}(q-1)^i -\sum_{i=1}^{k}\binom{m}{i}(q-1)^i&\geq\\
q^m-1-q^{m-1}+1 -2q^{2k} -2q^{2k}&\geq\\
q^m-5q^{m-1}&\geq 0,\\
\end{align*}
and Condition \eqref{Constraint1} holds.

Also,
$$\sum_{i=1}^{k}\binom{m}{i}(q-1)^{i-1}\leq \sum_{i=1}^{k}m^i(q-1)^{i-1}\leq \sum_{i=1}^{k}q^{2i-1} \leq 2q^{2k-1}\leq 2q^{m-2}\leq q^{m-1}-q^{m-2},$$
and Condition \eqref{Constraint2} is satisfied.

\end{proof}

\section{Acknowledgments}
The research of D. Bartoli was supported by Ministry for Education, University and Research of Italy (MIUR) (Project ``Geometrie di Galois e strutture di incidenza") and by the Italian National Group for Algebraic and Geometric Structures and their Applications (GNSAGA - INdAM). 

The research of M. Bonini was supported by the Italian National Group for Algebraic and Geometric Structures and their Applications (GNSAGA - INdAM).


\end{document}